\documentclass{amsart}

\newtheorem{theorem}{Theorem}[section]
\newtheorem{lemma}[theorem]{Lemma}

\theoremstyle{definition}
\newtheorem{definition}[theorem]{Definition}
\newtheorem{example}[theorem]{Example}
\usepackage{latexsym,amsfonts}
\numberwithin{equation}{section}

\def\co{\colon\thinspace}

\DeclareMathOperator{\Leib}{Leib}
\DeclareMathOperator{\alg}{alg\langle}
\DeclareMathOperator{\lat}{\mathcal{L}}
\begin{document}
\title{Lattices of subalgebras of Leibniz algebras}
\author{Donald W. Barnes}
\address{1 Little Wonga Rd, Cremorne NSW 2090 Australia}
\email{donwb@iprimus.com.au}

\begin{abstract}  I describe the lattice of subalgebras of a one-generator Leibniz algebra.  Using this, I show that, apart from one special case, a lattice isomorphism $\phi \co \lat(L) \to \lat(L')$ between Leibniz algebras $L,L'$maps the Leibniz kernel $\Leib(L)$ of $L$ to $\Leib(L')$.
\end{abstract}

\maketitle

\section{Introduction} \label{sec-intro}
A lot is known about the lattices of subalgebras of Lie algebras. See for example, Barnes \cite{LIsos}, \cite{Lautos},   Goto \cite{Goto}, Towers \cite{T-dimpres}, \cite{T-Lautos}, \cite{smod}. In this paper, I look at some basic results needed to extend these results to Leibniz algebras.

In the following, $L,L'$ are finite-dimensional (left) Leibniz algebras over the field $F$.  I denote by $\langle a, b, \dots\rangle$ the subspace spanned by the listed elements and by $\alg a,b, \dots \rangle$ the subalgebra they generate.  The Leibniz kernel of $L$ is the subspace $\Leib(L) = \langle x^2 \mid x \in L\rangle$ spanned by the squares of the elements of $L$.  By Barnes \cite[Lemma 1.1]{Sch-Leib}, it is a $2$-sided ideal and $L/\Leib(L)$ is a Lie algebra.

Since the main aim of this paper is to show that (apart from one exceptional case,) $\Leib(L)$ can be recognised from lattice properties alone, we say of a subalgebra $U$ of $L$, that $U$ is \textit{recognisable} if, from properties of the lattice $\lat(L)$, it can be shown that $U \subseteq \Leib(L)$, that is, if for every lattice isomorphism $\phi \co \lat(L) \to \lat(L')$, we have $\phi(U) \subseteq \Leib(L')$. 

There is one case in which $\Leib(L)$ is not recognisable.
\begin{example} \label{diamond} Let $L = \alg a \rangle = \langle a, a^2 \rangle$ with $a^3 = a^2$.  Put $b = a - a^2$ and $v = a^2$.  Then $b^2 = v^2 =vb = 0$ and $bv = v$.  Then $\Leib(L) = \langle v \rangle$ and $\lat(L)$ is

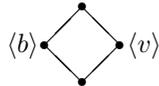
\begin{center}
\setlength{\unitlength}{1mm}
\begin{picture}(16,10)(-5,0)
\put(5,10){\circle*{1}} \put(0,5){\circle*{1}} \put(10,5){\circle*{1}} \put(5,0){\circle*{1}} 
\put(0,5){\line(1,1){5}} \put(5,0){\line(1,1){5}}
\put(0,5){\line(1,-1){5}} \put(5,10){\line(1,-1){5}}
\put(-5,4){$\langle b \rangle$}
\put(11,4){$\langle v \rangle$}
\end{picture}
\end{center}
The map $\phi \co \lat(L) \to \lat(L)$ which interchanges $\langle b \rangle$ and $\langle v \rangle$ is clearly a lattice automorphism.
\end{example}

The lattice of Example \ref{diamond} will be called the \textit{diamond lattice} and an algebra with this as its subalgebra lattice will be called a \textit{diamond algebra.}  The lattice of all subspaces of a vector space over $F$ will be called a \textit{vector space lattice}.

\begin{lemma}  Let $D$ be a diamond algebra.  Then there exist $b,v \in D$ such that $D = \langle b,v\rangle$ and $b^2 = v^2 = vb = 0$ and $bv = v$.  
\end{lemma}

\begin{proof} A Leibniz algebra $L$ which is not a Lie algebra has a proper subalgebra $\Leib(L)$, so an algebra without proper subalgebras is a 1-dimensional Lie algebra.  $D$ has two proper subalgebras, $B$ and $V$ say.  As not every 1-dimensional subspace of $D$ is a subalgebra, $D$ is not a Lie algebra, so one of these, $V$ say is $\Leib(L)$.  Take generators $b, v$ of $B, V$ respectively.  Then $b^2 = v^2 = vb=0$ and $bv = \lambda v$ for some $\lambda \in F$, $\lambda \ne 0$.  We replace $b$ with $b/\lambda$.
\end{proof}

To avoid repeated descriptions of notations, I use the following convention.  Whenever we have a subalgebra $U$ of $L$, $U'$ denotes the subalgebra $\phi(U)$ of $L'$ under the lattice isomorphism $\phi \co L \to L'$.  For an element $a \in L$, $a'$ denotes a generator of $\alg a \rangle'$.  (We shall see later, Lemma \ref{lem-kchains},  that $\alg a\rangle'$ is a one-generator algebra.  This holds trivially if $\dim\alg a\rangle = 1$.)  If $A \supseteq B$ are subalgebras of $L$, I denote by $A \div B$ the lattice interval consisting of all subalgebras $C$ with $A \supseteq C \supseteq B$.

\section{One-generator algebras}
 Consider the one-generator Leibniz algebra $L = \langle a, a^2, \dots, a^n \rangle$.   We have here a vector space $V = \langle a^2, \dots, a^n \rangle$ acted on by a  linear transformation $\theta : V \to V$ such that $V$ is generated as $F[\theta]$-module by the element $a^2$.  Let $f(x)= x^r g(x)$ be the characteristic polynomial of $\theta$.  Then $f(x)$ is also the minimum polynomial of $\theta$.  Put $V_1 = \theta^rV$.  Then $g(x)$ is the minimum polynomial of $\theta|{V_1}$.  

Put $h(x) = \bigl( g(x) - g(0)\bigr)/xg(0)$ and $b= a + h(\theta)a^2$.  Then
$$b^{r+2}= a^{r+2} + \theta^{r+1}h(\theta)a^2 = a^{r+2} + \bigl(g(\theta)-g(0)\bigr)a^{r+2}/g(0) = 0.$$
Let $B$ be the subalgebra generated by $b$.  Then $B^2$ is the only maximal subalgebra of $B$ and $\lat(B)$ consists of $B$ and all the subspaces of $B^2$.  If $a$ is not nilpotent, then $B$ is a proper subalgebra of $L$ not contained in the maximal subalgebra $V$, so $V$ is not the only maximal subalgebra of $L$.

\begin{lemma}\label{nilp1gen}  Let $B = \alg b\rangle$ be a subalgebra generated by a nilpotent element $b$.   Then $B \simeq B'$ and $B^2 = \Leib(B)$ is recognisable.

\begin{proof} $B^2$ is the only maximal subalgebra of $B$, so $W = \phi(B^2)$ is the only maximal subalgebra of $B'$.  There exists $c \in B'$ , $c \notin W$.  Since $c$ is not contained in any maximal subalgebra of $B'$, we have $\alg c \rangle = B'$.  Since $W$ is the only maximal subalgebra of $B'$, $c$ is nilpotent.  Since the maximal chains of $\lat(B)$ and $\lat(B')$ have the same length, we have $B \simeq B'$. $\Leib(B) = B^2$ and $\phi(\Leib(B)) = W = \Leib(B')$.
\end{proof}  
\end{lemma}

\begin{lemma} \label{ifnilp} Suppose $c \notin V$ is nilpotent, then $\alg c \rangle = B$.
\end{lemma}

\begin{proof} Since $c \notin V$, we have $c = \lambda b + b_1 + v$ for some $\lambda \in F$, $\lambda \ne 0$, $b_1 \in B^2$ and $v \in V_1$.  Since $\theta$ is non-singular on $V_1$, $v = 0$ and $c \in B$.  But $c$ is not in the only maximal subalgebra of $B$, so $\alg c \rangle = B$.
\end{proof} 

To determine the invariant subspaces, we use the prime power factorisation $g(x) = p_1^{r_1}(x) \dots p_k^{r_k}(x)$ where the $p_i(x)$ are distinct irreducible polynomials.  Since $V$ and so also $V_1$ is generated as $F[\theta]$-module by a single element, the only invariant subspaces of $V_1$ are the spaces
$$V_{s_1, \dots, s_k} = \{v \in V \mid p_1^{s_1}(\theta) \dots p_k^{s_k}(\theta) v = 0\} = \theta^r p_1^{r_1-s_1}(\theta) \dots p_k^{r_k-s_k}(\theta) V,$$
for $s_i \le r_i$.  Put $U_{s_1, \dots, s_k} = B + V_{s_1, \dots, s_k}$.  

\begin{theorem}  Let $L = \alg a \rangle$.  Then the $U_{s_1, \dots, s_k}$ are the only subalgebras of $A$ not contained in $V$.  The lattice interval $L \div B$ is the lattice product of $k$ chains of lengths $r_1, \dots, r_k$.
\end{theorem}

\begin{proof}  Let $C$ be a subalgebra of $L$ not contained in $V$.  Then $C$ has an element $c = b + b_1 + v$ where $b_1 \in B^2$ and $v \in V_1$.  As above , we obtain a nilpotent element $c_0 = c+w \in C$ where $w \in V$.  By Lemma \ref{ifnilp}, $c_0$ generates $B$.  Thus $B \subseteq C$.  It follows also that  $v \in C$, and the invariant subspace generated by $v$ is contained in $C$.  It follows that $C = U_{s_1, \dots, s_k}$ for some $s_1, \dots, s_k$.  That the lattice interval is the product of chains as described follows.
\end{proof}

To illustrate this, I consider the case $r=0$, $k=1$, $r_1=1$ and $p_1(x) = x-1$.

\begin{example}\label{ex-schain} Let $L = \langle b, v_1, v_2\rangle$ with $b^2 =  v_ix = 0$ for all  $x \in L$ and $bv_1 = v_1+v_2$, $bv_2 = v_2$.  Then $L = \alg b+v_1\rangle$ and $\lat(L)$ is
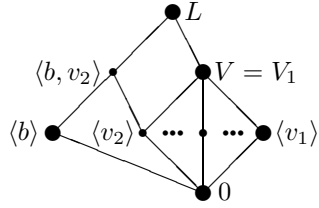
\begin{center}
\setlength{\unitlength}{1mm}
\begin{picture}(30,30)(-15,-2)
\put(8,16){\circle*{2}} \put(0,8){\circle*{1}} \put(16,8){\circle*{2}} \put(8,0){\circle*{2}} 
\put(0,8){\line(1,1){8}} \put(8,0){\line(1,1){8}}
\put(0,8){\line(1,-1){8}} \put(8,16){\line(1,-1){8}}
\put(8,8){\circle*{1}} \put(8,0){\line(0,1){16}}
\put(3,8){\circle*{0.7}} \put(4,8){\circle*{0.7}} \put(5,8){\circle*{0.7}}
\put(13,8){\circle*{0.7}} \put(11,8){\circle*{0.7}} \put(12,8){\circle*{0.7}}
\put(-7,7){$\langle v_2 \rangle$} \put(17.5,7){$\langle v_1 \rangle$} \put(9.5,15){$V = V_1$} \put(10,-1){$0$}
\put(-12,8){\circle*{2}} \put(8,0){\line(-5,2){20}}
\put(-12,8){\line(1,1){16}} \put(4,24){\circle*{2}} 
\put(-4,16){\circle*{1}} 
\put(4,24){\line(1,-2){4}} \put(-4,16){\line(1,-2){4}}
\put(5.5,23){$L$} \put(-18,7){$\langle b \rangle$} 
\put(-14.5,15){$\langle b,v_2 \rangle$} 
\end{picture}\end{center}
\end{example}
  
Observe that the emphasised points $L,V_1, \langle b\rangle, \langle v_1\rangle, 0$ form a sublattice, and that $\lat(L)$ is not modular.  Indeed, for any one-generator algebra with $\dim(V_1) > 1$, taking $v_1$ an element which generates $V_1$ under the action of $\theta$, we obtain in this way the standard non-modular lattice.

\begin{definition} The \textit{signature} of the one-generator Leibniz algebra $L$ is the list $[r|r_1, \dots r_k|d_1, \dots, d_k]$ where $d_i$ is the degree of the irreducible polynomial $p_i(x)$.  If $k=1$, we call $L$ a \textit{single-chain algebra}.
\end{definition}
Clearly, from the signature and knowledge of the field $F$, one can reconstruct $\lat(L)$. The algebra $L$ of Example \ref{ex-schain} is a single-chain algebra with signature $[0|2|1]$.  

\begin{lemma} \label{lem-schain}  Suppose that $L$ is a single-chain algebra.  Then $L'$ is a single-chain algebra with the same signature as $L$.  If $\dim(L) > 2$, then $\phi(\Leib(L)) = \Leib(L')$.
\end{lemma}

\begin{proof}  $L$ has exactly two maximal subalgebras, so $L'$ has exactly two maximal subalgebras.  A vector space cannot be the set union of two proper subspaces, so there exists $a' \in L'$ which is not contained in any maximal subalgebra.  Thus $\alg a' \rangle = L'$. Let $[r|r_1|d_1]$ be the signature of $L$.  Let $M$ be maximal subalgebra containing $B$.

The lattice of $V$ is the vector space lattice of dimension $\dim(V) = r+r_1d_1$ and it follows that $\dim(L') \ge 1+r+r_1d_1$, while chains in $\lat(M)$ have length at most $r+r_1d_1$.  If $\Leib(L') = M'$, then $M$ also has the $(r+ r_1d_1)$-dimensional vector space lattice.  But $M$ has at most two maximal subalgebras, one containing $B$ and $M \cap V$.  This requires $\dim(M) = 1$ and $\dim(L)=2$.  In this case, $L \simeq L'$.  If $\dim(L)>2$, then $\Leib(L') = V'$, and the signature of $L'$ can be read from the length of the chain $L' \div B'$ and the dimensions of $B' \cap V'$ and $V'$.
\end{proof}

\begin{lemma}\label{lem-kchains}  Let $L$ be a one-generator Leibniz algebra and suppose that the number $k$ of chains is greater than $1$.  Then $L'$ is a one-generator algebra with the same signature as $L$ and $\phi(\Leib(L)) = \Leib(L')$.
\end{lemma}

\begin{proof}  Let $[r|r_1, \dots, r_k|d_1, \dots, d_k]$ be the signature of $L$.  For each $i$, we have a single-chain subalgebra $C_i \supset B$ with signature $[r|r_i|d_i]$ such that $L \div B$ is the product of the chains $C_i\div B$. 

I prove first that $B' \not\subseteq \Leib(L')$.  For this, it is sufficient to consider the case $k=2$.  If $r>0$ or for any $i$, we have $r_id_i>1$, then by Lemma \ref{lem-schain}, $B' \not\subseteq \Leib(L')$. 
So suppose that  $r=0$ and that $r_i = d_i = 1$ for $i=1,2$.  Then $\lat(L)$ is
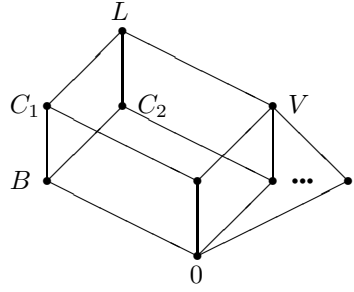
\begin{center}
\setlength{\unitlength}{1mm}
\begin{picture}(50,40)(0,-5)
\put(20,30){\circle*{1}} \put(18.5,31.5){$L$}
\put(10,20){\circle*{1}} \put(5,19){$C_1$}
\put(20,20){\circle*{1}} \put(22,19){$C_2$}
\put(40,20){\circle*{1}} \put(42,19){$V$}
\put(10,20){\line(1,1){10}}  \put(20,20){\line(0,1){10}}
\put(20,30){\line(2,-1){20}}
\put(10,10){\circle*{1}} \put(5,9){$B$}
\put(10,10){\line(0,1){10}}  \put(10,10){\line(1,1){10}}
\put(30,10){\circle*{1}} \put(40,10){\circle*{1}} 
\put(20,20){\line(2,-1){20}} \put(10,20){\line(2,-1){20}} 
\put(30,10){\line(1,1){10}} \put(40,10){\line(0,1){10}} 
\put(30,0){\circle*{1}} \put(29,-3.5){$0$}
\put(10,10){\line(2,-1){20}} \put(30,0){\line(0,1){10}} 
\put(30,0){\line(1,1){10}}
\put(30,0){\line(2,1){20}} 
\put(50,10){\circle*{1}} \put(40,20){\line(1,-1){10}}
\put(43,10){\circle*{0.7}} \put(44,10){\circle*{0.7}} \put(45,10){\circle*{0.7}}
\end{picture}\end{center}
As $C_1, C_2$ do not have vector space lattices, $\Leib(L')$ cannot be $C'_1$ or $C'_2$.  If $B' \subseteq \Leib(L')$, then $B' = \Leib(L')$ and $L'/B.$ is a Lie algebra, contrary to $L' \div B'$ being the diamond lattice.

We now have, whatever the signature of $L$, that $\phi(\Leib(B))$, $\phi(\Leib(C_i))$ are all contained in $\Leib(L')$.  But they generate $V'$, so $V' \subseteq \Leib(L)$.  As $V'$ is a maximal subalgebra of $L'$, this implies that $V' = \Leib(L')$.

Take an element $b'$ which generates $B'$ and let $\beta\co V' \to V'$ be left multiplication by $b'$.  Then the minimum polynomial of $\beta|C'_i \cap V'$ is $x^rq^{r_i}_i(x)$ for some irreducible polynomial $q_i(x)$.  I now prove that the $q_i$ are distinct.

Suppose that $q_1=q_2$.  Let $W_i$ be a minimal invariant subspace of $C'_i \cap V'$ with minimum polynomial $q_i(x)$.  Then $W_1 \simeq W_2$ as $F[\beta]$-modules.  Take an isomorphism $\gamma \co W_1 \to W_2$ and put $W^* = \{w+\gamma(w) \mid w \in W_1 \}$.  Then $W^* $ is an invariant subspace and $B'+W^*$ is a subalgebra not in the product of chains.  Therefore $q_1 \ne q_2$.

We now have that $\beta^rV'$ is generated as $F[\beta]$-module by some single element $w$, and it follows that $\alg b'+w \rangle = L'$.
\end{proof}

\section{Recognising $\Leib(L)$}
In this section, $L,L'$ are Leibniz algebras and $\phi \co \lat(L) \to \lat(L')$ is a lattice isomorphism.  The aim is to prove that $\phi(\Leib(L)) = \Leib(L')$.  For this to fail, we must have a diamond subalgebra $\langle b,v\rangle$, $bv = v$, with $\langle b \rangle'\subseteq \Leib(L')$.  I shall assume this and show that, if $\dim(L) \ge 3$, then there are other subalgebras whose relation to $\langle b, v\rangle$ makes this impossible.  It is convenient to represent data as a geometric configuration, with points representing 1-dimensional subalgebras and lines representing 2-dimensional subalgebras, all of whose subspaces are subalgebras.  Thus the lines represent Lie subalgebras.  Broken lines are used to represent 2-dimensional subalgebras which have subspaces that are not subalgebras.  Their lattice of subalgebras is the diamond lattice, the product of two chains each of length 1.

\begin{theorem}  Let $L,L'$ be Leibniz algebras and let $\phi \co \lat(L) \to \lat(L')$ be a lattice isomorphism.  Suppose $\dim(L) \ge 3$.  Then $\phi(\Leib(L)) = \Leib(L'))$.

\begin{proof}  In the notation set out above, I assume that $\langle b \rangle' \subseteq \Leib(L')$.  I investigate and eliminate a number of cases.

\textit{Case 1:}  Suppose that there exists $x \in L$, $x^2 = xb=bx=xv=vx=0$.  Since $(b+\lambda x)v=v$, we have that $\langle (b+\lambda x), v\rangle$ is a diamond subalgebra for all $\lambda \in F$.
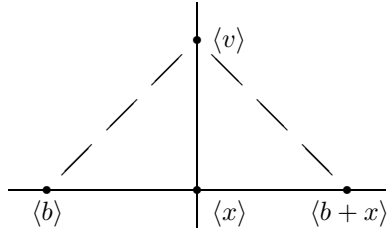
\begin{center} \setlength{\unitlength}{1mm}
\begin{picture}(100,40)
\put(50,30){\circle*{1}} \put(50,10){\circle*{1}} 
\put(30,10){\circle*{1}} \put(70,10){\circle*{1}} 
\multiput(52,28)(6,-6){3}{\line(1,-1){4}}
\multiput(48,28)(-6,-6){3}{\line(-1,-1){4}}
\put(25,10){\line(1,0){50}}
\put(50,5){\line(0,1){30}}
\put(52,29){$\langle v \rangle$}
\put(52,6){$\langle x \rangle$}
\put(28,6){$\langle b \rangle$}
\put(65,6){$\langle b+x \rangle$}
\end{picture}
\end{center}

For suitable choice of $b' $ and $x'$, we have $b'+x' \in \langle b+x\rangle'$. Since $b' \in \Leib(L)$, we can choose $v'$ such that $v'b'=b'$.  Since $\langle v', b'+x'\rangle$ is a diamond algebra, we must have $v'(b'+x') = \lambda(b'+x')$ for some $\lambda \in F$.   But $b', (b'+x') \in \Leib(L')$, so $x' \in \Leib(L')$.  Since $\langle v', x' \rangle$ is a Lie algebra,  $v'x' = - x'v'=0$ and $v'(b'+x') = b'$ contrary to $v'(b'+x') = \lambda(b'+x')$.  Thus Case 1 is impossible.

\textit{Case 2:}  Suppose that $\dim(\Leib(L)) > 1$.  Then there exists $w \in \Leib(L)$ not in $\langle v\rangle$.  If in the space $W$ generated by $w$ under the action $\theta$ of $b$ contains an element $w_0$ such that $\theta w_0 = 0$, then we have Case 1.  Therefore $\theta$ acts non-singularly on $W$ and $\alg b+w\rangle \supset W$.  If $\dim(W) >1$, then we cannot have $b' \in \Leib(L')$, so $bw=\lambda w$, $\lambda \ne 0$.  
If $\lambda \ne 1$, then for every $\mu \ne 0$, $b$ and $v+\mu w$ generate a 3-dimensional subalgebra.  If $\lambda = 1$, then for all $\mu$, $b$ and $v+\mu w$ generate a diamond algebra.
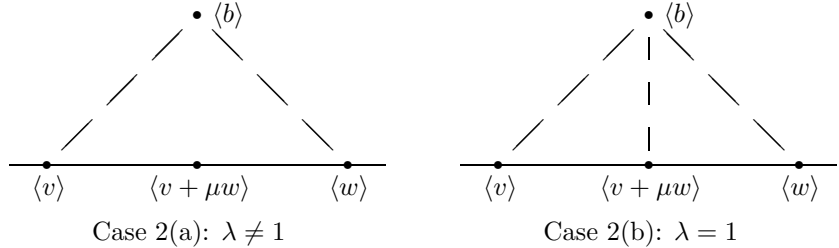
\begin{center} \setlength{\unitlength}{1mm}
\begin{picture}(120,40)
\put(30,30){\circle*{1}} \put(30,10){\circle*{1}} 
\put(10,10){\circle*{1}} \put(50,10){\circle*{1}} 
\multiput(32,28)(6,-6){3}{\line(1,-1){4}}
\multiput(28,28)(-6,-6){3}{\line(-1,-1){4}}
\put(5,10){\line(1,0){50}}
\put(32,29){$\langle b \rangle$}
\put(23.5,6){$\langle v+\mu w \rangle$}
\put(8,6){$\langle v \rangle$}
\put(47.5,6){$\langle w \rangle$}
\put(16,0){Case 2(a): $\lambda \ne 1$}
\put(90,30){\circle*{1}} \put(90,10){\circle*{1}} 
\put(70,10){\circle*{1}} \put(110,10){\circle*{1}} 
\multiput(92,28)(6,-6){3}{\line(1,-1){4}}
\multiput(88,28)(-6,-6){3}{\line(-1,-1){4}}
\put(65,10){\line(1,0){50}}
\put(92,29){$\langle b \rangle$}
\put(83.5,6){$\langle v+\mu w \rangle$}
\put(68,6){$\langle v \rangle$}
\put(107.5,6){$\langle w \rangle$}
\multiput(90,12)(0,6){3}{\line(0,1){3}}
\put(76,0){Case 2(b): $\lambda= 1$}
\end{picture}
\end{center}
Since $b' \in \Leib(L')$,  for suitable choice of $v', w'$, we have $v'b' = b'$ and $w'b'=b'$.  But this implies that $(v'-w')$ and $b'$ generate a 2-dimensional Lie algebra, contrary to the lattice information.  Therefore Case 2 is impossible.

\textit{Case 3:} $\dim(\Leib(L)) = 1$.  Take $x \notin \langle b,v \rangle$.  If $x^2 \ne 0$ but $x^3 = 0$, then $\phi(\langle x^2 \rangle )\subseteq \Leib(L')$ contrary to assumption.  Therefore either $x^2 = 0$ or $\langle x, v\rangle$ is the diamond algebra.  In either case, there exists $c = x + \lambda v$ with $c^2 = 0$.  Either $cv=v$ or $cv=0$.  Since $v' \notin \Leib(L')$, if $cv = v$, then $c' \in \Leib(L')$ and $\dim(\Leib(L')) >1$, contrary to Case 2 applied to $\phi^{-1}$.  So $cv=0$.  But this implies that  $x^2 = xv= 0$.  Therefore, for every $x \notin \langle b,v \rangle$, we have $xv=0$.  But $xv=0$ and $(x+b)v = 0$ implies $bv=0$ contrary to assumption.  Thus Case 3 also is impossible.
\end{proof} 
\end{theorem}
\bibliographystyle{amsplain}

\end{document}